\documentclass{amsart}
\usepackage{enumitem}
\usepackage{graphicx}
\usepackage{url}

\usepackage[backref=page,colorlinks,citecolor=blue,bookmarks=true]{hyperref}
\hypersetup{linkcolor=blue,  citecolor=blue, urlcolor=blue}
\usepackage{amsrefs}
\usepackage{microtype}

\theoremstyle{theorem}
\newtheorem{theorem}{Theorem}
 
\theoremstyle{definition}
\newtheorem*{definition}{Definition}

\begin{document}

\title[Stabilizer and zero-in-the-corner problems]{Undecidability of the stabilizer and zero-in-the-corner problems for matrix groups}
\author[E.\ Breuillard]{Emmanuel Breuillard} \address{Emmanuel Breuillard\hfill\break 	Mathematical Institute \hfill\break Oxford OX1 3LB, United Kingdom} \email{breuillard@maths.ox.ac.uk}
\author[G. Kocharyan]{Georgi Kocharyan} \address{Georgi Kocharyan\hfill\break 	St Catherine College,  \hfill\break Cambridge  CB2 1RL, United Kingdom} \email{gk445@cam.ac.uk}
\maketitle
\begin{abstract}
The upper-left-corner problem and upper-right-corner problem for matrix semigroups with integer entries are known to be undecidable. Here we ask the analogous problems for matrix \textit{groups} and prove their undecidability for rational entries, by reducing to the undecidability of the word problem for groups. To reach this aim, we answer a question of Dixon from 1985 by proving the undecidability of the stabilizer problem for matrix groups.
\end{abstract}

\noindent


\section{Introduction.}
Few developments in mathematics have had such far-reaching theoretical \textit{and} practical consequences as the first algorithmically undecidable problem posed by Turing in 1936. This served as a springboard for proving the undecidability of many celebrated problems appearing naturally in mathematics, such as the existence of integer solutions to diophantine equations, or the word problem for finitely presented groups. In this note we establish the following undecidability results for subgroups of invertible matrices:

\begin{theorem}\label{ULCP} Let $n \ge 6$ and $G$ be a finitely generated subgroup of  $\mathrm{GL}_n(\mathbb{Q})$. The problem of whether $G$ contains a matrix with a zero entry in its upper left corner is undecidable.
\end{theorem}
\begin{theorem} \label{URCP}
Let $n \geq 9$ and $G$ be a finitely generated subgroup of $\mathrm{GL}_n(\mathbb{Q}).$ The problem of whether $G$ contains a non-identity matrix with a zero entry in its upper right corner is undecidable.
\end{theorem}

In other words, there is no algorithm that upon receiving as input a finite set $X$ of invertible matrices with rational entries, is capable of deciding whether or not the subgroup $\langle X \rangle$ they generate contains a matrix with zero $(1,1)$-entry, and analogously a non-identity one with zero $(1,n)$-entry.

These problems are inspired by the analogous questions for matrix semigroups, which have been known to be undecidable for a long time already for the ring of $3\times 3$-matrices with integer entries $M_3(\mathbb{Z})$, see \cite{paterson1970, manna, halava2001, cassaigne2018}. In fact for each entry $(i,j)$ of the matrix, we have a well-defined $(i,j)$-problem, asking for the existence of a semigroup element with vanishing $(i,j)$-entry. It is easy to see that the on-diagonal problems are all equivalent and so are the off-diagonal problems. It turns out that both are known to be undecidable, see \cite{halava2001}.


For the semigroup problems the proof is based on one of the very first provably undecidable problems in word combinatorics, the so-called \emph{Post Correspondence Problem} or PCP.  The PCP  asks to determine if given a family of dominoes, each bearing some string of zeros and ones on its top tile and also on its  bottom tile, one can lay them down in some order in such a way that the same string is read on the top and on the bottom.  More formally, if $\Sigma_2$ denotes the free semigroup on two letters, the problem  asks whether a finitely generated semigroup of the direct product $\Sigma_2 \times \Sigma_2$ contains an element from the diagonal $\{(x,x) \in \Sigma_2\times \Sigma_2, x \in \Sigma_2\}$. It turns out that one can embed $\Sigma_2 \times \Sigma_2$ inside $M_3(\mathbb{Z})$ in such a way that an element lies in the diagonal if and only if the $(3,2)$-entry of the corresponding matrix vanishes, see \cite[p. 63]{manna}. This allows to encode the PCP inside the $(3,2)$-problem. The on-diagonal problem is dealt with similarly.

In this note, we consider the upper-left-corner problem for matrix groups and show that it is undecidable for at least $17$ matrices of size at least $6$. We also prove this for the upper-right-corner problem for at least $17$ matrices of size at least $9$. The problem for groups cannot be reduced to semigroups. Of course every group is a semigroup, so clearly these problems are related, but the wrong way around! If we had an algorithm for semigroups, we would have one for groups. But as the undecidability of the semigroup version is known \cite{halava2001}, there is no direct consequence for the group version. This gives rise to a host of important differences. For instance:
\begin{itemize}[leftmargin=10pt]
\item In the semigroup case the problems over $\mathbb{Q}$ and $\mathbb{Z}$ were clearly equivalent. 
This is no longer the case for invertible matrices. Here we will only deal with rational entries.
\item The off-diagonal problem needs to be tweaked a little to make sense in the group setting, because obviously the identity matrix has a zero off-diagonal entry. It should ask for a \emph{non-identity} matrix in $G$ with an off-diagonal zero.

\end{itemize}
Such generalizations of semigroup problems to group problems tend to be quite difficult. For example, there is a statement of the PCP for groups \cite{myasnikov2014}, but its decidability is currently an open problem! Two closely related problems were advertised by Dixon in 1985 for matrix groups. The \emph{orbit problem} asks, given a finitely generated subgroup $G$ of $\mathrm{GL}_n(\mathbb{Q})$ and two vectors $u$ and $v$ in $\mathbb{Q}^n,$ whether some element of $G$ maps $u$ onto $v$. Dixon \cite{dixon1985} proved that the orbit problem is undecidable, by reducing to it the membership problem in the product of two free groups (see below). The second problem, the so-called \emph{stabilizer problem}, asks to decide whether $u$ is fixed by some non-identity element of $G$. The question of its undecidability was left open until now. We settle it here:

\begin{theorem} \label{stabilizer}
Let $n \ge 9.$ The stabilizer problem on $\mathbb{Q}^n$ is undecidable.
\end{theorem}

In fact we will prove that the analogous problem over $\mathbb{Z}$ for $G$ in $\mathrm{GL}_n(\mathbb{Z})$ is already undecidable. The main idea in proof of Theorem \ref{stabilizer} will be essential for the proof of Theorem \ref{URCP}.
Below we offer geometrically flavored proofs for the undecidability of the upper-left-corner problem (Theorem \ref{ULCP}), upper-right-corner problem (Theorem \ref{URCP}) and the stabilizer problem (Theorem \ref{stabilizer})  in  $\mathrm{GL}_n(\mathbb{Q})$. While the PCP was used for the analogous problems on semigroups, we shall use the \textit{word problem} for groups as our source of undecidability. There will be a common feature however in that, in place of a free semigroup embedding, we shall construct a suitable embedding of the free group $F_2$ on $2$ generators inside $\mathrm{GL}_3(\mathbb{Q})$  that will enable us to encode the word problem for groups within the two problems we are considering. We shall require two commuting copies of $F_2$ for Theorem  \ref{ULCP} and three commuting copies for Theorems \ref{URCP} and \ref{stabilizer}.


\section{Word problem and membership problem.}
Let $F_k$ denote the free group on $k$ letters. Any group $G$ generated by $k$ elements can be given a \emph{presentation} as a quotient of $F_k$ in the form:
\begin{equation}\label{pres}G=\langle x_1,\ldots,x_k| r_1=\ldots=r_j=\ldots\rangle\end{equation}
where the $r_j$ are words in $F_k$ called \emph{relators}.
\begin{definition}[Word problem]
The \textit{word problem} for a finitely generated group $G$ asks for an algorithm that given any word $w\in F_k$ in the generators, decides whether or not $w = 1$ in $G.$
\end{definition}
The classic Novikov-Boone theorem (e.g. \cite{rotman1995}) states the existence of a \textit{finitely presented group} for which the word problem is undecidable. The proof is constructive and gives the generators and relations explicitly. The proof of the undecidability of the word problem for groups is much more involved than its counterpart for semigroups, which came earlier in works of Post and Turing \cite{turing1950}. Borisov \cite{borisov1969} gives an example of a group with just 5 generators and 12 relations that has undecidable word problem. Analyzing the original paper by Borisov and a discussion by Collins \cite{collins1972} shows that the group constructed by Borisov is  a tower of repeated HNN-extensions and free products with amalgamation starting with $F_2.$ Considering the normal form of a word in such an extension guaranteed by Britton's Lemma, we see that the only way for an element in the extension to have finite order is if its corresponding element in the original group did -- but $F_2$ is torsion-free. This implies Borisov's group is also torsion-free, a fact that we will use later.
\begin{definition}[Membership problem]
The \textit{membership problem} or \textit{generalized word problem} on a finitely generated group and a set of group elements  $X = \{ w_1, w_2 \ldots \}$ asks for an algorithm that given $g$ in the group decides whether or not $g \in \langle X \rangle.$
\end{definition}
That this is undecidable already for the group $F_k \times F_k$, $k \geq 2$, follows from a clever observation due to Mihailova \cite{mihailova1968}. 
 She reduced the problem to the undecidability of the word problem by means of a fiber product construction, now known as the \textit{Mihailova construction.}
\begin{definition}[Mihailova construction]
Let $G$ be a group generated by $k$ elements. The \textit{Mihailova subgroup} $M(G)$ of $F_k \times F_k$ is 
$$\{ (x,y) \in F_k \times F_k | \; x = y \text{ in } G \}.$$
\end{definition}
It is easy to check that if $G$ is given by a presentation as in $(\ref{pres})$, with $k$ generators and $r$ relators, then $M(G)$ is generated by $k+r$ elements, namely the $(x_i,x_i)$ and the $(1,r_i)$ (see e.g. \cite[Chap. III]{miller1971}). But now the following is immediate.
\begin{theorem}[\cite{mihailova1968}]\label{mem}
Let $n \geq 2.$ The membership problem on $F_n \times F_n$ is undecidable.
\end{theorem}
\begin{proof}
Since $F_n$ embeds into $F_2,$ it is enough to prove this for $n=2.$ Take a finitely presented group $G$ on $k$ generators with undecidable word problem. Since $F_k \times F_k$ embeds into $F_2 \times F_2,$ $M(G)$ can be viewed as a subgroup of $F_2 \times F_2.$ But $w = 1$ in $G$ if and only if $(w,1) \in M(G)$ which is finitely generated, so an algorithm for the membership problem on $F_2 \times F_2$ would give one for the word problem on $G.$
\end{proof}
We now propose a variation on the Mihailova construction that will lend itself towards statements about centralizers. With the same notation as before, i.e. considering $M(G)$ as a subgroup of $F_2 \times F_2$, we define
$$M_1(G) = \{ (x,y,z) \in F_2 \times F_2 \times F_2| \; x = y, (y,z)  \in M(G) \}. $$ As before, $M_1(G)$ is finitely generated (by  the $(x_i,x_i,x_i)$ and $(1,1,r_i)$) and if $G$ is the group constructed by Borisov, both $M(G)$ and $M_1(G)$ are generated by 17 elements. The usefulness of defining $M_1(G)$ is captured by the following:
\begin{theorem}\label{cent}
Write $F_2 = \langle a,b \rangle$ and pick $w \in F_2 \setminus\{1\}.$ The centralizer of $(a,b,w) \in F_2 \times F_2 \times F_2$ has non-trivial intersection with $M_1(G)$ if and only if $w$ has finite order in $G.$
\end{theorem}
\begin{proof}
Let $(m,m,n) \in M_1(G)$ commute with $(a,b,w).$ As $(m,n) \in M(G),$ by definition we see that $m = n$ in $G.$ Now $m$ commutes with both $a$ and $b$ in $F_2,$ which implies $m = 1$ in $F_2,$ so $n = 1$ in $G.$ So if $(m,m,n)$ is non-trivial, we must have $n \neq 1$ in $F_2.$ Further, looking at the third component we see $w$ commutes with $n$ in $F_2.$ By the Nielsen-Schreier theorem, this means $\langle n, w \rangle$ is free and abelian, thus cyclic. Let $ \langle w_0 \rangle = \langle n, w \rangle.$ As we know $n \neq 1$ in $F_2,$ this implies $w_0^k = n$ in $F_2$ for some $k \neq 0.$ Hence $w_0^k = 1$ in $G,$ and thus $w$ as a power of $w_0$ is also torsion in $G.$  \\
Conversely, let $w^k = 1$ in $G$ with $k \neq 0$ and $w \neq 1$ in $F_2.$ Then $(1,1,w^k) \in M_1(G)$ commutes with $(a,b,w)$ and is non-trivial. 
\end{proof}

\section{The upper-left-corner problem on matrix groups is undecidable.}

The upper-left-corner problem for $n \times n$ rational matrices can be understood as asking whether or not the vector $\vec{e}_1=(1,0, \ldots,0) \in \mathbb{Q}^n$ can be mapped to its orthogonal complement $\langle \vec{e}_2,\ldots,\vec{e}_n\rangle$ by some element from $\langle X \rangle.$ After applying an appropriate change of basis, the problem reduces to the:
\begin{definition}[External hyperplane problem]
Given a finite subset $X$ in $\mathrm{GL}_n(\mathbb{Q}),$ the \textit{hyperplane problem} asks for an algorithm that decides, given a hyperplane $H \subseteq \mathbb{Q}^n$ and a vector $\vec{v} \in \mathbb{Q}^n$ with $\vec{v} \not\in H,$ whether or not there is $M \in \langle X \rangle$ such that $M \vec{v} \in H.$
\end{definition}
Given an instance of the external hyperplane problem, conjugating by an appropriate matrix will replace $\vec{u}$ with $\vec{e}_1$ and $H$ with $\langle \vec{e}_2,\ldots,\vec{e}_n\rangle.$ The input will then correspond to a conjugate subgroup of $\langle X \rangle,$ and we could directly solve the problem by a hypothetical algorithm for the upper-left-corner problem for groups.
So if we can prove the external hyperplane problem in $\mathbb{Q}^n$ to be undecidable, the same will be true for the upper-left-corner problem on $\mathrm{GL}_n(\mathbb{Q}).$ We are now going to set about to do just that. To this end, we shall construct a suitable embedding of the free group $F_2$ inside $\mathrm{GL}_3(\mathbb{Q})$ and our first step is the following:


\begin{theorem} \label{freesubgr}
There exists a free subgroup $F_2$ of $\mathrm{SL}_2(\mathbb{Z})$ on two generators such that every non-identity element $A$ is \textit{hyperbolic}, i.e. $|\mathrm{Tr}(A)| > 2.$
\end{theorem}
\begin{proof}
The following matrices generate such a free subgroup:
\begin{align*}
\begin{pmatrix}
3 & 2 \\
1 & 1
\end{pmatrix}, \begin{pmatrix}
1 & 1 \\
2 & 3
\end{pmatrix}
\end{align*}
Any pair of hyperbolic matrices $A, B \in \mathrm{SL}_2(\mathbb{R})$ such that $C = ABA^{-1}B^{-1}$ fulfills $\mathrm{Tr}(C) < -2$ would do. A short and elegant proof of this fact can be found in a paper by Purzitsky \cite[Theorem 8]{purzitsky1972} and relies on basic hyperbolic geometry: $\mathrm{PSL}_2(\mathbb{R})$ is the group of isometries of the Lobachevsky upper-half plane and the condition $\mathrm{Tr}(C) < -2$ is equivalent to the requirement that $C$ is hyperbolic and that the axes of $A$ and $B$ intersect. Then $A,B$ will generate a Schottky group, in particular it will be  discrete, free and all its non-identity elements will be  hyperbolic. In fact, almost every pair of matrices would do. See \cite{aoun} for a general result of this kind.
\end{proof}

We now embed this free group inside $\mathrm{SL}_3(\mathbb{Z})$ in such a way that its image preserves an affine cone in $\mathbb{R}^3$.

\begin{theorem} \label{main}
There exists a free subgroup $F$ of $\mathrm{SL}_3(\mathbb{Z})$ on two generators, a rational linear form $f: \mathbb{Q}^3 \to \mathbb{Q}$ and a vector $\vec{u} \in \mathbb{Q}^3$ with the property that for all $g \in F:$
\begin{enumerate}[label=(\arabic*),leftmargin=18pt]
\item $f(g\vec{u}) \geq 0$ and
\item $f(g\vec{u}) = 0 \Leftrightarrow g = 1$.
\end{enumerate}
\end{theorem}
\begin{proof}
Consider the surface in three dimensions given by $x^2+yz = 0.$ This is the union of the origin and two disjoint open cones defined by the conditions $y > z$ and $y < z$ respectively, see Figure \ref{cone}. We firstly construct a free subgroup on two generators of matrices in $\mathrm{SL}_3(\mathbb{Z})$ that preserves each of the two cones separately. To achieve this, we interpret a vector on the cone as a matrix in $ \begin{pmatrix}
x & y \\ z & -x
\end{pmatrix} \in M_2(\mathbb{R})$ with zero determinant and zero trace. The conjugation action of $\mathrm{SL}_2(\mathbb{R})$ on $M_2(\mathbb{R})$ preserves these two properties, so keeps the vector on the double cone.\\

\noindent \textbf{Claim:} This action preserves each cone $ \{ y < z \} $ and $ \{ y > z \}.$

\smallskip
\noindent \emph{Proof:} It is easy to calculate this directly by computing the effect of the conjugation action on vectors with $y < z$ and $y > z$ respectively. Alternatively, a quicker way is to notice that $\mathrm{SL}_2(\mathbb{R})$ is connected, so if $\vec{u} \neq 0$ is a vector on the cone and $g \in \mathrm{SL}_2(\mathbb{R}),$ there is a path between $\vec{u}$ and $g\vec{u}$ that lies in $\mathrm{SL}_2(\mathbb{R})\vec{u}.$ But as every such $g$ is invertible, such a path cannot pass through the origin, proving the claim. \\

This gives a representation of $\mathrm{SL}_2(\mathbb{R})$ into $\mathrm{GL}_3(\mathbb{R})$ which we can explicitly calculate as follows:
Associate $\begin{pmatrix}
x & y \\ z & -x
\end{pmatrix}$ with $(x,y,z)$ and the matrix corresponding to conjugation by $\begin{pmatrix}
a & b \\ c & d
\end{pmatrix}$ can be calculated by looking at its effect on the basis vectors 
\begin{align*}
\begin{pmatrix}
1 & 0 \\ 0 & -1
\end{pmatrix}, \begin{pmatrix}
0 & 1 \\ 0 & 0
\end{pmatrix}, \begin{pmatrix}
0 & 0 \\ 1 & 0
\end{pmatrix}.
\end{align*}
Checking what conjugation does to these matrices, we see that the representation can be written as
\begin{align}
\phi: \mathrm{SL}_2(\mathbb{R}) &\to \mathrm{GL}_3(\mathbb{R}) \nonumber \\
\begin{pmatrix}
a & b \\ c & d
\end{pmatrix} &\mapsto \begin{pmatrix}
ad +bc & -ac & bd \\
-2ab & a^2 & -b^2 \\
2cd & -c^2 & d^2
\end{pmatrix} \tag{\textdagger}
\end{align}
We check that $\phi$ is a group homomorphism with values in $\mathrm{SL}_3(\mathbb{R})$. Also $$\phi(\mathrm{SL}_2(\mathbb{Z})) \subset \mathrm{SL}_3(\mathbb{Z}).$$

Set  $\vec{u} = (0,1,0).$
\begin{figure} [h!]
 \centering
 \includegraphics[width=0.6\textwidth]{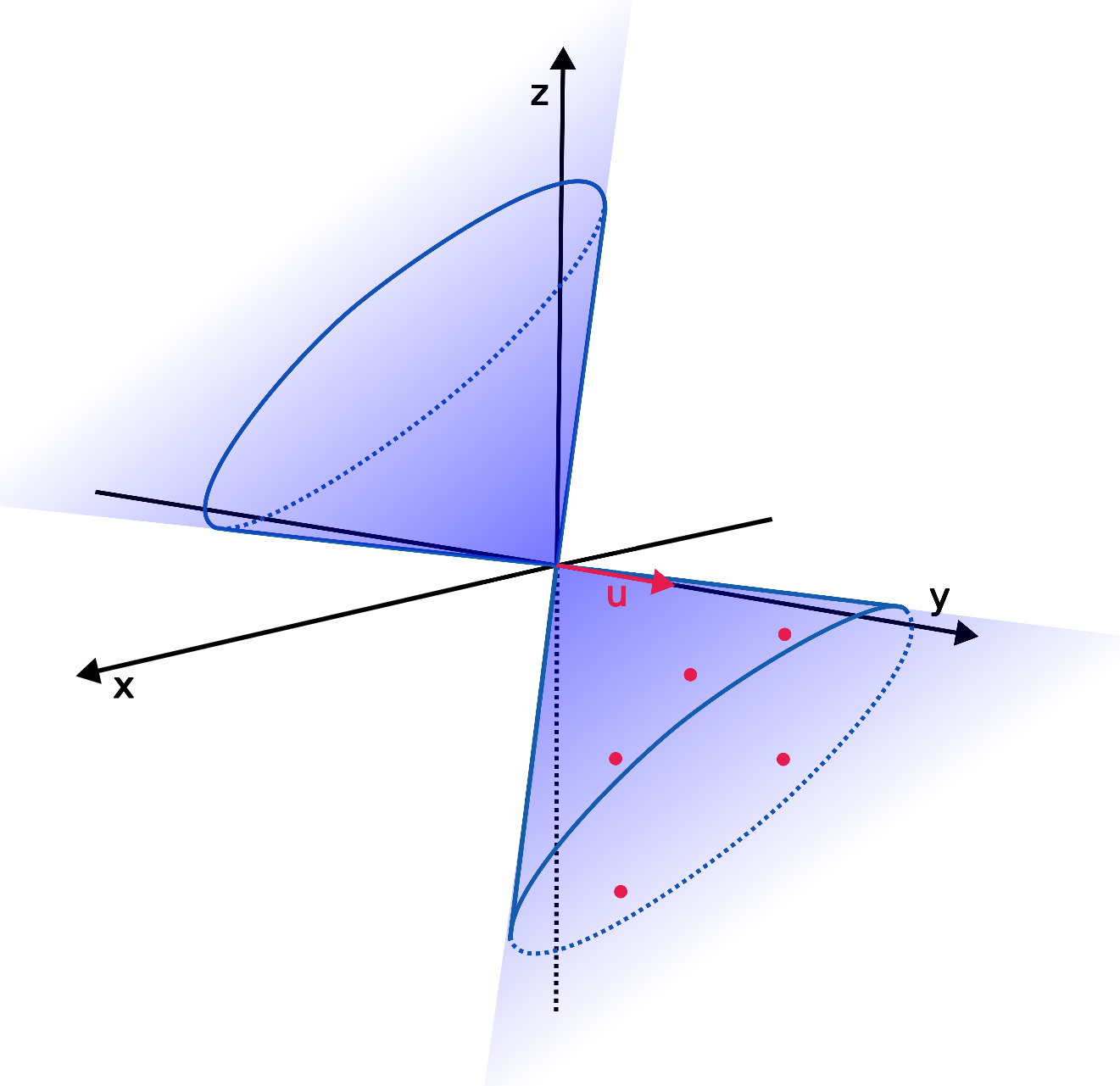}
    \caption{The double cone with the vector $\vec{u}$ and some elements of its orbit under $F.$}  \label{cone}
    \end{figure} We have $Q(\vec{u}) = 0$ for the quadratic form $Q(\vec{x}) = -(x^2+yz)$ whose zero set is our double cone. The associated symmetric bilinear form $B$ is $B(\vec{x}, \vec{y}) = \frac{1}{2}(Q(\vec{x}) + Q(\vec{y}) - Q(\vec{x}-\vec{y})).$
We now claim that the linear form given by $f(\vec{x}) = B(\vec{u}, \vec{x})$ and the hyperplane $H = \{ \vec{x} \; | \; B(\vec{u}, \vec{x}) = 0 \}$ fulfill the required properties.
Simplifying the equation for the hyperplane we get
\begin{align*}
B(\vec{u}, \vec{x}) = 0 \Leftrightarrow
Q(\vec{x}) = Q(\vec{x}-\vec{u}) \Leftrightarrow
x^2+yz = x^2 + (y-1)z \Leftrightarrow
z = 0.
\end{align*}
So $H$ is the $xy$-plane: this is precisely the plane tangent to the double cone that contains $\vec{u}.$ Let $F=\phi(F_2)$ where $F_2\leq \mathrm{SL}_2(\mathbb{Z})$ is the subgroup given by Theorem \ref{freesubgr}. We claim that $F, \vec{u}$ and $f$ fulfill the conditions. Pick a $g \in F.$ Now  by the previous calculation, $f(g\vec{u}) > 0$ holds if and only if
\begin{align*}
B(\vec{u}, g\vec{u}) > 0
\Leftrightarrow (g\vec{u})_z < 0,
\end{align*}
which simply means that $g\vec{u}$ is on the same side of the hyperplane as the cone that $\vec{u}$ is in -- but because each cone is preserved by $g,$ this always happens for any $g$ unless $g\vec{u}$ is proportional to $\vec{u},$ in which case $f(g\vec{u}) = 0.$ So condition $(1)$ is shown. 
Now we want to show that if $f(g\vec{u}) = 0,$ it must be that $g = 1.$
We are able to explicitly calculate this using $(\dagger).$ If $\vec{u}$ is mapped to a multiple of itself then
\begin{align*}
\begin{pmatrix}
ad +bc & -ac & bd \\
-2ab & a^2 & -b^2 \\
2cd & -c^2 & d^2
\end{pmatrix}\begin{pmatrix}
0 \\ 1 \\ 0
\end{pmatrix} = \begin{pmatrix}
0 \\ q \\ 0
\end{pmatrix} \Rightarrow \begin{pmatrix}
-ac \\ a^2 \\ -c^2
\end{pmatrix} = \begin{pmatrix}
0 \\ q \\ 0
\end{pmatrix} \Rightarrow c = 0, a = \pm\sqrt{q}.
\end{align*}
Since $g=\phi(\gamma)$ with $\gamma \in \mathrm{SL}_2(\mathbb{Z})$ we must have $\det{\gamma} = ad = 1,$ thus $q = \pm 1$, $(a,d) = \pm (1,1)$ and $\mathrm{Tr}(\gamma) = \pm 2.$ But by design all non-identity elements in $F_2$ are hyperbolic, i.e. have trace with modulus larger than 2. We conclude that $g = 1.$ 
\end{proof}
Exploiting the key feature of the above construction that the orbit of $\vec{u}$ remains on one side of the hyperplane, we can now finalize the proof of Theorem \ref{ULCP}.
\begin{proof}[Proof of Theorem \ref{ULCP}]
Without loss of generality we can assume that $n=6.$ As stated above, it suffices to prove the undecidability of the external hyperplane problem. We will reduce it to the undecidability of the membership problem for $F_2 \times F_2,$ so our starting point is a finite set $X$ in $F_2 \times F_2$ for which we know the membership problem to be undecidable (such $X$ is constructed explicitly in Theorem \ref{mem} and with the group from \cite{borisov1969} this gives an $X$ with $|X|\leq 17$). We pick two copies of each of the objects whose existence was guaranteed by Theorem \ref{main}, say $F^{(1)}, F^{(2)}, \vec{u_1}, \vec{u_2}$ and $f_1, f_2,$ and we define $$\vec{u} = \begin{pmatrix}
\vec{u}_1 \\ \vec{u}_2
\end{pmatrix} \in \mathbb{Z}^6.$$ Let $H \subseteq \mathbb{Q}^6 $ be the hyperplane defined as the kernel of the linear form $f_1(\vec{x}) + f_2(\vec{y})$, where $ (\vec{x}, \vec{y})  \in \mathbb{Q}^6.$   Now  $F_2 \times F_2 \simeq F^{(1)} \times F^{(2)}$ embeds in $\mathrm{GL}_6(\mathbb{Q})$ via the map  $$(g^{(1)},g^{(2)}) \mapsto \begin{pmatrix}
g^{(1)} & 0 \\
0 & g^{(2)}
\end{pmatrix}. $$
Suppose we are given an instance of the membership problem $g = (g^{(1)},g^{(2)}) \in F_2 \times F_2,$ for which we would like to decide whether or not it is in $\langle X \rangle.$
A hypothetical algorithm for the external hyperplane problem would allow to decide this. Indeed we would ask whether or not some group element $h$ from $\langle X \rangle$ exists that  maps $\vec{v} := g^{-1}\vec{u}$ into $H.$ By positivity of the $f_i$ (property (1)), this happens if and only if $f_i(hg^{-1}u_i) = 0$ for both $i = 1,2.$ By property (2) of $f_1,$ this holds if and only if $hg^{-1}=1$, i.e. if $g \in \langle X \rangle.$ This is in contradiction with Theorem \ref{mem} and ends the proof.
\end{proof}
\section{The upper-right-corner problem on matrix groups is undecidable.}

Recall Dixon's \textit{stabilizer problem} \cite{dixon1985} mentioned in the introduction:
\begin{definition}[Stabilizer problem]
Given a finitely generated subgroup $G$ of $\mathrm{GL}_n(\mathbb{Q})$ and $u \in \mathbb{Q}^n,$ the \textit{stabilizer problem} asks whether  $\mathrm{Stab}_X(u) = \{ g \in G| \; gu = u \}$ is the trivial group.
\end{definition}
We now prove that this problem is undecidable (Theorem \ref{stabilizer}). In fact we prove that the analogous problem over $\mathbb{Z}$ is already undecidable. Later, we will use a crucial step in the proof to prove the undecidability of the upper-right-corner problem.
\begin{proof}[Proof of Theorem \ref{stabilizer}]
Let $G$ be Borisov's torsion-free finitely presented group that has undecidable word problem, and $M_1(G) \leq F_2 \times F_2 \times F_2$ the variant of the Mihailova construction described in Section 2, which is 17-generated. Embed $F_2$ into $\mathrm{SL}_2(\mathbb{Z}),$ so $M_1(G)$ acts by conjugation on triples of $2 \times 2$-matrices with zero trace, which is a vector space of dimension $9.$ The stabilizer of $(2a-\mathrm{Tr}(a)I_2, 2b-\mathrm{Tr}(b)I_2, 2c-\mathrm{Tr}(c)I_2$ is the centralizer of $(a,b,c) \in F_2 \times F_2 \times F_2.$
In particular, this vector has a non-trivial stabilizer in $M_1(G)$ if and only if the centralizer of $(a,b,c)$ in $F_2 \times F_2 \times F_2$ has non-trivial intersection with $M_1(G).$ But we saw in Theorem \ref{cent} that this happens if and only if the  word $c$ has finite order in $G.$ So we could use a hypothetical algorithm solving the stabilizer problem (over $\mathbb{Q}$ or even $\mathbb{Z}$) to decide if a word in $G$ has finite order or not. Since our $G$ is torsion-free, this would mean being able to solve the word problem on it.
\end{proof}
Combining ideas from the proofs of Theorem \ref{main} and Theorem \ref{stabilizer}, we shall now prove the undecidability of the upper-right-corner problem on matrix groups.  This reduces to a problem very similar to the external hyperplane problem: 
\begin{definition}[Internal hyperplane problem]
Given a finite subset $X$ in $\mathrm{GL}_n(\mathbb{Q}),$ the \textit{internal hyperplane problem} asks to decide, given a hyperplane $H \subseteq \mathbb{Q}^n$ and a vector $\vec{v} \in \mathbb{Q}^n$ with $\vec{v} \in H,$ whether or not there is $M \in \langle X \rangle\setminus\{1\}$ such that $M \vec{v} \in H.$
\end{definition}
To prove the undecidability of the internal hyperplane problem, we will essentially use the same geometric argument as before, but it will require some extra work. 
\begin{proof}[Proof of Theorem \ref{URCP}]
Without loss of generality we may take $n = 9.$ We saw the following in the proof of the undecidability of the stabilizer problem: there is a certain finitely generated subgroup $\Gamma$ of $F_2 \times F_2 \times F_2$ (namely $M_1(G)$ from before), for which no algorithm exists that given an element $w \in F_2 \times F_2 \times F_2$ will decide whether the intersection of its centralizer with $\Gamma,$ denoted by $C_\Gamma(w),$ is trivial or not. This is the problem we wish to reduce the internal hyperplane problem to. 

To this end, we again pick an $F_2$ in $\mathrm{SL}_2(\mathbb{Z})$ as guaranteed in Theorem \ref{freesubgr}, i.e. one with all non-identity elements hyperbolic. Then, we realize it in $\mathrm{GL}_3(\mathbb{Q})$ via the same representation $\phi$ as before. The aim before was to pick a $\vec{u}$ on a cone and make the group act on it in such a way that it would only remain in the hyperplane only if the acting group element satisfied a certain property (in that case, was trivial). Now, we will need a bit more care in selecting the vector. Given a fixed $A \in \mathrm{SL}_2(\mathbb{Z}),$ we will pick it as follows: The matrix $\phi(A)$ has at least one eigenvector $\vec{v}_A=(x,y,z)$ with eigenvalue $\lambda \neq 1,$ as $A$ is hyperbolic. This vector lies on the cone $x^2+yz = 0$ from before. Indeed if $$M= \begin{pmatrix}
x & y \\ z & -x
\end{pmatrix}$$ and $AMA^{-1} = \lambda M$ with $\lambda \neq 1$, then $\det M = 0.$ Now, let $\vec{u} = \vec{v}_A$ and observe that $A$ and any matrix in the centralizer  of $A$ must fix $\mathbb{R}\vec{v}_A.$ We want to show that the converse also holds, i.e. any matrix fixing $\mathbb{R}\vec{v}_A$ must also commute with $A.$ This follows from a handy fact: 

\smallskip
\noindent \textbf{Claim:} Let $B_0 \subset \mathrm{SL}_2(\mathbb{R})$ be the subgroup of upper-triangular matrices. The set $\{ M \in \mathrm{SL}_2(\mathbb{R})| \; \phi(M)\vec{v}_A \in \mathbb{R}\vec{v}_A \}$ is a conjugate subgroup of $B_0$ that contains $A.$

\smallskip
\noindent \emph{Proof:} Since $\mathrm{SL}_2(\mathbb{R})$ acts transitively on the cone via $\phi$, conjugating appropriately, without loss of generality we may assume that $A \in B_0, \vec{v}_A = (0,1,0).$ The claim follows instantly after computing the stabilizer by hand: it is precisely $B_0.$ \\

Now, this means that the centralizer $C_{F_2}(A)$ of $A$ in $F_2$  is contained in a conjugate subgroup of $B_0,$  which we call $B$. Note that $B_0$ (and hence $B$) is solvable (its derived subgroup is abelian). So by the Nielsen-Schreier theorem, $F_2 \cap B$ is free and solvable and hence cyclic, generated say by $M$. This implies $$C_{F_2}(A) \leq F_2 \cap B = \langle M \rangle \Rightarrow A = M^k$$ for some $k \in \mathbb{Z}\setminus\{0\}.$ As $A$ and $M$ clearly commute, we see that in fact $C_{F_2}(A) = \langle M \rangle = F_2 \cap B.$ So letting $H$ be the hyperplane tangent to the cone along the new $\vec{u} = \vec{v}_A$ this time, we obtain a situation similar to that of Theorem \ref{main}: for a fixed $A \in F_2$ we get a free subgroup $F$ of $\mathrm{SL}_3(\mathbb{Z})$ on two generators, a rational linear form $f: \mathbb{Q}^3 \to \mathbb{Q}$ and a vector $\vec{v}_A \in \mathbb{Q}^3$ with the property that for all $g \in F:$
\begin{enumerate}[label=(\arabic*),leftmargin=18pt]
\item $f(g\vec{v}_A) \geq 0$ and
\item $f(g\vec{v}_A) = 0 \Leftrightarrow g \in C_{F_2}(A)$.
\end{enumerate}
Finally, we will prove that an algorithm for the internal hyperplane problem would be able to detect the triviality of $C_\Gamma(w)$ as described above. 
Similarly as before, we shall pick three copies of the objects just constructed, with each component of $w= (A,A',A'')$ playing the role of $A$ above: $F^{(1)}, F^{(2)}, F^{(3)},$ with corresponding linear forms $f_1,f_2,f_3$ and vectors $\vec{v}_A, \vec{v}_{A'}, \vec{v}_{A''}.$
Now  as before $F_2 \times F_2 \times F_2 \simeq F^{(1)} \times F^{(2)} \times F^{(3)}$ embeds in $\mathrm{GL}_9(\mathbb{Q})$ via $$(g^{(1)},g^{(2)},g^{(3)}) \mapsto \begin{pmatrix}
g^{(1)} & 0 & 0 \\
0 & g^{(2)} & 0 \\
0 & 0 & g^{(3)}
\end{pmatrix}, $$
and we let $H$ be the kernel of $f_1(\vec{x}) + f_2(\vec{y}) + f_3(\vec{z}),$ where $(\vec{x} , \vec{y}, \vec{z})  \in \mathbb{Q}^9.$ By the positivity properties of the $f_i,$ a non-identity matrix from $\Gamma$ maps $(\vec{v}_A, \vec{v}_{A}, \vec{v}_{A''}) \in H$ to $H$ if and only if $f_1(\vec{v}_A) = f_2(\vec{v}_{A'}) = f_3(\vec{v}_{A''}) = 0.$ But by the construction of $f,$ this occurs if and only if $C_X(w)$ is non-trivial. This ends the proof.
\end{proof}

\noindent \emph{Final remarks}. It is natural to ask to what extent the $n\ge 6$ (resp. $n \ge 9$) requirement on the size of the matrices can be reduced, or if we can take fewer than $17$ matrices. Whether a single matrix would suffice in the case of the upper-right-corner problem is a reformulation of the so-called bi-Skolem problem \cite{halava2005,biskolem}. Its undecidability would imply that of the original Skolem problem. Also, it would be interesting to know whether Theorems \ref{ULCP} and \ref{URCP} continue to hold for integer matrices (Theorem \ref{stabilizer} does by our proof), and to determine whether one can improve these results by showing that no algorithm exists that given a finitely generated subgroup of $\mathrm{GL}_n(\mathbb{Q})$ can decide whether it contains a matrix with at least one zero entry on the diagonal, or respectively a non-identity matrix with at least one zero entry. \\

\paragraph{\textbf{Acknowledgement.}}
The authors thank the Department of Pure Mathematics and Mathematical Statistics at the University of Cambridge for funding the summer research project that led to this work.

\bibliographystyle{vancouver}

\bibliography{monthlyrefs.bib}

\end{document}